\documentclass[english]{amsart}
\usepackage[T1]{fontenc}
\usepackage[latin9]{inputenc}
\usepackage{mathrsfs}
\usepackage{amsthm}
\usepackage{amsmath}
\usepackage{amssymb}

\makeatletter
\theoremstyle{plain}
\newtheorem{thm}{\protect\theoremname}
  \theoremstyle{plain}
  \newtheorem{lem}[thm]{\protect\lemmaname}

\makeatother

\usepackage{babel}
  \providecommand{\lemmaname}{Lemma}
\providecommand{\theoremname}{Theorem}

\begin{document}

\title[Nondegeneracy of critical points of the mean curvature]{Nondegeneracy of critical points 
of the mean curvature of the boundary for Riemannian manifolds}

\author{Marco Ghimenti}
\address[Marco Ghimenti] {Dipartimento di Matematica,
  Universit\`{a} di Pisa, via F. Buonarroti 1/c, 56127 Pisa, Italy}
\email{marco.ghimenti@dma.unipi.it.}
\author{Anna Maria Micheletti}
\address[Anna Maria Micheletti] {Dipartimento di Matematica,
  Universit\`{a} di Pisa, via F. Buonarroti 1/c, 56127 Pisa, Italy}
\email{a.micheletti@dma.unipi.it.}

\begin{abstract}
Let $M$ be a compact smooth Riemannian manifold of finite dimension
$n+1$ with boundary $\partial M$and $\partial M$ is a compact $n$-dimensional
submanifold of $M$. We show that for generic Riemannian metric $g$,
all the critical points of the mean curvature of $\partial M$ are
nondegenerate.
\end{abstract}

\date{\today}

\maketitle

Let $M$ be a connected compact $C^{m}$ manifold with $m\ge3$ of
finite dimension $n+1$ with boundary $\partial M$. The boundary
$\partial M$ is a compact $C^{m}$ $n$-dimensional submanifold of
$M$. Let $\mathscr{M}^{m}$ be the set of all $C^{m}$ Riemannian
metrics on $M$. Given a metric $\bar{g}\in\mathscr{M}^{m}$, we consider
the mean curvature of the boundary $\partial M$ of $(M,\bar{g})$. 
Our goal is to prove that generically for a Riemannian metric $g$
all the critical points of the mean curvature of the boundary $\partial M$
of $(M,g)$ are nondegenerate. More precisely we show the following
result
\begin{thm}
\label{thm:main}The set 
\begin{equation}
{\mathscr{A}}=\left\{ \begin{array}{c}
g\in\mathscr{M}^{m}:\text{ all the critical points of the}\\
\text{mean curvature of the boundary}\\
\text{of }(M,g)\text{ are nondegenerate}
\end{array}\right\} \label{eq:A}
\end{equation}
is an open dense subset of $\mathscr{M}^{m}$.
\end{thm}
We denote by $\mathscr{S}^{m}$ the space of all $C^{m}$ symmetric
covariant $2$-tensors on $M$ and $\mathscr{B}_{\rho}$ the ball
in $\mathscr{S}^{m}$ of radius $\rho$. The set $\mathscr{M}^{m}$
of all $C^{m}$ Riemannian metrics on $M$ is an open subset of $\mathscr{S}^{m}$. 

A possible application of this result arises in the study of the following
Neumann problem
\begin{equation}
\left\{ \begin{array}{cc}
-\varepsilon\Delta_{g}u+u=u^{p-1} & \text{in }M\\
u>0 & \text{in }M\\
{\displaystyle \frac{\partial u}{\partial\nu}=0} & \text{on }\partial M
\end{array}\right.\label{eq:P}
\end{equation}
where $p>2$ if $n=2$ and $2<p<2^{*}=\frac{2n}{n-2}$ if $n\ge3$,
$\nu$ is the external normal to $\partial M$ and $\varepsilon$
is a positive parameter. In \cite{BP} the authors prove that the
problem (\ref{eq:P}) has a mountain pass solution $u_{\varepsilon}$
which has a unique maximum point $\xi_{\varepsilon}\in\partial M$
converging, as $\varepsilon\rightarrow0$, to a maximum point of the
mean curvature of the boundary. Recently, in \cite{GM} a relation
between topology of the boundary $\partial M$ and the the number
of solutions is established. More precisely it has been proved that
Problem (\ref{eq:P}) has at least $\text{cat }\partial M$ non trivial
solution provided $\varepsilon$ small enough. In a forthcoming paper
\cite{G} the author shows that nondegenerate critical points of the
mean curvature of the boundary produce one peak solutions of problem
(\ref{eq:P}).

In our opinion the role of the mean curvature of the boundary in Problem
(\ref{eq:P}) on a manifold $M$ with boundary is similar to the role
of the scalar curvature in the problem 
\[
\left\{ \begin{array}{cc}
-\varepsilon\Delta_{g}u+u=u^{p-1} & \text{in }M\\
u>0 & \text{in }M
\end{array}\right.
\]
defined on a boundaryless manifold $M$. Recently in some papers \cite{DMP,MP09a,MP09b}
assuming a sort of non degeneracy of the critical points of the scalar
curvature of the boundaryless manifold $(M,g)$ some results of existence
of one peak and multipeak solutions have been proved. Moreover, in
\cite{MP10} it is proved that generically with respect to the metric
$g$ all the critical points of the scalar curvature are nondegenerate,
so the latter result of non degeneracy can be applied to the previous existence theorems.

\section{Setting of the problem}

In the following, with abuse of notation we often identify a point
in the manifold with its Fermi coordinates $(x_{1},\dots,x_{n},t)$.
We now recall the definition of Fermi coordinates.

Let $\bar{\xi}$ belongs to the boundary $\partial M$, let $\left(x_{1},\dots,x_{n}\right)$
be coordinates on the $n$ manifold $\partial M$ in a neighborhood
of a point $\bar{\xi}$. Let $\gamma(t)$ be the geodesic leaving
from $\bar{\xi}$ in the orthogonal direction to $\partial M$ and
parametrized by arc lenght. Then the set $ $$\left(x_{1},\dots,x_{n},t\right)$
are the so called \emph{Fermi coordinates} at $\bar{\xi}\in\partial M$
where $\left(x_{1},\dots,x_{n}\right)\in B_{\bar{g}}(0,R)$ and $0\le t<T$
for $R,T$ small enough. 

In these coordinates the arclenght $ds^{2}$ is written as 
\[
ds^{2}=dt^{2}+g_{ij}(x,t)dx_{i}dx_{j}\text{ for }i,j=1,\cdots,n.
\]
Also, we set $I(\bar{\xi},R)$ the neighborhood of $\bar{\xi}$ such
that, in Fermi coordinates, $|x|=|(x_{1},\dots,x_{n})|<R$ and $0\le t<R.$

If $\bar{g}$ is the metric of the manifold $M$ then $\text{det }\bar{g}=\mbox{\text{det }}(g_{ij})_{ij}$.

We denote by $h_{ij}(x,t)$ the second fundamental form of the submanifold
\[
\partial M_{t}=\left\{ (x,t)\ :\ x\in\partial M,\ 0\le t<T\right\} 
\]
for $T$ small enough. Moreover, $H^{\bar{g}}(x,t)$ is the trace
of the second fundamental form $h_{ij}(x,t)$ of the submanifold $\partial M_{t}$,
that is 
\begin{equation}
H^{\bar{g}}(x,t)=\left(\bar{g}(x,t)\right)^{ij}h_{ij}(x,t)\label{eq:traccia}
\end{equation}
By a well known result of Escobar \cite{Es92} we have that the second
fundamental form in a neighborhood of a point $\bar{\xi}\in\partial M$
can be expressed in term of the metric $\bar{g}$ of the manifold
in the following way 
\begin{equation}
\partial_{t}\bar{g}_{ij}(x,t)=-2h_{ij}(x,t).\label{eq:escobar}
\end{equation}
where $(x,t)$ are the Fermi coordinates centered at $\bar{\xi}$.

We denote by $\mathscr{S}^{m}$ the Banach space of all $C^{m}$ symmetric
covariant symmetric $2$-tensors on $M$. The norm $||\cdot||_{m}$
is defined in the following way. We fix a finite covering $\{V_{\alpha}\}_{\alpha\in L}$
of $M$ such that the closure of $V_{\alpha}$ is contained in $U_{\alpha}$
where $\{U_{\alpha},\psi_{\alpha}\}$ is an open coordinate neighborhood.
Let $V_{\alpha}\cap\partial M=\emptyset$. If $k\in\mathscr{S}^{m}$
we denote by $k_{ij}$ the components of $k$ with respect to the
coordinates $(x_{1},\dots,x_{n+1})$ on $V_{\alpha}$. We define 
\[
||k||_{m}=\sum_{\alpha\in L}\sum_{|\beta|\leq m}\sum_{i,j=1}^{n+1}
\sup_{\psi_{\alpha}(V_{\alpha})}\frac{\partial^{\beta}k_{i,j}}{\partial x_{1}^{\beta_{1}}\cdots\partial x_{m}^{\beta_{m}}}
\]
We proceed in analogous way when $V_{\alpha}\cap\partial M\ne\emptyset$,
by means of Fermi coordinates.

We recall and abstract result in transversality theory (see \cite{Qu70,ST79,Uh76})
which will be fundamental for our results. 
\begin{thm}
\label{thm:trans}Let $X,Y,Z$ be three real Banach spaces and let
$U\subset X,\ V\subset Y$ be two open subsets. Let $F$ be a $C^{1}$
map from $V\times U$ in to $Z$ such that 
\begin{enumerate}
\item For any $y\in V$, $F(y,\cdot):x\rightarrow F(y,x)$ is a Fredholm
map of index $0$. 
\item $0$ is a regular value of $F$, that is $F'(y_{0},x_{0}):Y\times X\rightarrow Z$
is onto at any point $(y_{0},x_{0})$ such that $F(y_{0},x_{0})=0$. 
\item The map $\pi\circ i:F^{-1}(0)\rightarrow Y$ is $\sigma$-proper,
that is $F^{-1}(0)=\cup_{s=1}^{+\infty}C_{s}$ where $C_{s}$ is a
closed set and the restriction $\pi\circ i_{|C_{s}}$ is proper for
any $s$. \\
Here $i:F^{-1}(0)\rightarrow Y\times X$ is the canonical embedding
and $\pi:Y\times X\rightarrow Y$ is the projection.
\end{enumerate}
then the set $\theta=\left\{ y\in V\ :\ 0\text{ is a regular value of }F(y,\cdot)\right\} $
is a residual subset of $V$, that is $V\smallsetminus\theta$ is
a countable union of closed subsets without interior points.
\end{thm}
At this point we introduce the $C^{1}$ map 
\begin{eqnarray}
F & : & \mathscr{B}_{\rho}\times B(0,R)\rightarrow\mathbb{R}^{n}\nonumber \\
F(k,x) & = & \left.\nabla_{x}H^{\bar{g}+k}(x,t)\right|{}_{(x,0)}\label{eq:F}
\end{eqnarray}
where $H^{\bar{g}+k}(x,t)$ is the mean curvature of $\partial M_{t}$
related to the metric $\bar{g}+k$ at the point $(x,t)$. This map
is $C^{1}$ if $m\ge3$. Moreover, by (\ref{eq:traccia}) and (\ref{eq:escobar})
we have 
\begin{equation}
F(k,x)=-\frac{1}{2}\left.\nabla_{x}\left(\left(\bar{g}+k\right)^{ij}(x,,t)\partial_{t}\left(\bar{g}+k\right)_{ij}(x,,t)\right)\right|_{(x,0)}\label{eq:Fbis}
\end{equation}

\begin{lem}
\label{lem:aperto}The set 
\[
{\mathscr{A}}=\left\{ \begin{array}{c}
g\in\mathscr{M}^{m}:\text{ all the critical points of the}\\
\text{mean curvature of the boundary}\\
\text{of }(M,g)\text{ are nondegenerate}
\end{array}\right\} 
\]
is an open set in $\mathscr{M}^{m}.$\end{lem}
\begin{proof}
If $\bar{g}\in\mathscr{A}$, we have that the critical point of the
mean curvature of $\partial M$ are in a finite number, say $\xi_{1},\dots,\xi_{\nu}$.
We consider the Fermi coordinates in a neighborhood of $\xi_{1}$,
and the map $F$ defined in (\ref{eq:F}) and (\ref{eq:Fbis}).

We have that $F(0,0)=0$ and that $\partial_{x}F(0,0):\mathbb{R}^{n}\rightarrow\mathbb{R}^{n}$
is an isomorphism because $\xi_{1}$ is a nondegenerate critical point.
Thus by the implicit function theorem there exist two positive numbers
$\rho_{1}$ and $R_{1}$ and a unique function $x_{1}(k)$ such that
in $\mathscr{B}_{\rho_{1}}\times B(0,R_{1})$ the level set $\left\{ F(k,x)=0\right\} $
is the graphic of the function $\left\{ x=x_{1}(k)\right\} $. 

We can argue analogously for $\xi_{2},\dots,\xi_{\nu}$, finding constant
$\rho_{2},R_{2},\dots\rho_{\nu},R_{\nu}$ for which the set $\left\{ F(k,x)=0\right\} $
in a neighborhood of $\mathscr{B}_{\rho_{2}}\times B(0,R_{2}),\dots\mathscr{B}_{\rho_{\nu}}\times B(0,R_{\nu})$
can be respectively described by means of the functions $x_{2}(k),\dots x_{\nu}(k)$.

We set $B_{i}=\left\{ \xi\in\partial M\ :\ d_{g}(\xi,\xi_{i})<R_{i}\right\} $.
We claim that there are no critical points of the mean curvature for
the metric $\bar{g}+k$ in the set ${\displaystyle \partial M\smallsetminus\cup_{i=1}^{\nu}B_{i}}$
for any $k\in\mathscr{B}_{\rho}$, provided $\rho$ sufficiently small.
Otherwise we can find a sequence of $\left\{ \rho_{n}\right\} _{n}\rightarrow0$,
a sequence $k_{n}\in\mathscr{B}_{\rho_{n}}$ and a sequence of points
${\displaystyle \xi_{n}\in\partial M\smallsetminus\cup_{i=1}^{\nu}B_{i}}$
such that $F(k_{n},\xi_{n})=0$. But, by compactness of $\partial M$,
$\xi_{n}\rightarrow\bar{\xi}$ for some ${\displaystyle \bar{\xi}\in\partial M\smallsetminus\cup_{i=1}^{\nu}B_{i}}$
and, by continuity of $F$, $\bar{\xi}$ is such that $F(0,\bar{\xi})=0$,
that is a contradiction.

At this point the proof is complete.
\end{proof}

\section{Proof of the main theorem}

We are going to apply the transversality Theorem \ref{thm:trans}
to the map $F$ defined in (\ref{eq:F}). In this case we have $X=Z=\mathbb{R}^{n}$,
$Y=\mathscr{S}^{m}$, $U=B(0,R)$ and $V=\mathscr{B}_{\rho}$ with
$R$ and $\rho$ small enough. Since $X$ and $Z$ are finite dimensional,
it is easy to check that for any $k\in\mathscr{B}_{\rho}$ the map
$x\mapsto F(k,x)$ is Fredholm of index $0$, so assumption (i) holds. 

To prove assumption (ii) we will show, in Lemma \ref{lem:ii}, that,
if the pair $(\tilde{x},\tilde{k})\in B(0,R)\times\mathscr{B}_{\rho}$
is such that $F(\tilde{k},\tilde{x})=0$, the map $F'_{k}(\tilde{k},\tilde{x})$
defined by
\[
k\rightarrow D_{k}\left.\nabla_{x}H^{\bar{g}+\tilde{k}}(x,t)\right|_{(\tilde{x},0)}[k]
\]
 is surjective. 

As far as it concerns assumption (iii) we have that 
\[
F^{-1}(0)=\cup_{s=1}^{\infty}C_{s}\text{ where }C_{s}
=\left\{ \overline{B(0,R-1/s)}\times\overline{\mathscr{B}_{\rho-1/s}}\right\} \cap F^{-1}(0).
\]
It is easy to check that the restriction $\left.\pi\circ i\right|_{C_{s}}$
is proper, that is if the sequence $\left\{ k_{n}\right\} _{n}\subset\overline{\mathscr{B}_{\rho-1/s}}$
converges to $k_{0}$ in $\mathscr{S}^{m}$ and the sequence $\left\{ x_{n}\right\} _{n}\subset\overline{B(0,R-1/s)}$
is such that $F(k_{n},x_{n})=0$ then by compactness of $\overline{B(0,R-1/s)}$
there exists a subsequence of $\left\{ x_{n}\right\} _{n}$ converging
to some $x_{0}\in\overline{B(0,R-1/s)}$ and $F(k_{0},x_{0})=0.$

So we are in position to apply Theorem \ref{thm:trans} and we get
that the set 
\begin{eqnarray*}
A(\bar{\xi},\rho) & = & \left\{ \begin{array}{c}
k\in\mathscr{B}_{\rho}\ :\ F'_{k}(x,k):\mathbb{R}^{n}\rightarrow\mathbb{R}^{n}\text{ is onto}\\
\text{at any point }(x,k)\text{ s.t. }F(x,k)=0
\end{array}\right\} \\
 & = & \left\{ \begin{array}{c}
k\in\mathscr{B}_{\rho}\ :\ F'_{k}(x,k):\mathbb{R}^{n}\rightarrow\mathbb{R}^{n}\text{ is invertible}\\
\text{at any point }(x,k)\text{ s.t. }F(x,k)=0
\end{array}\right\} \\
 & = & \left\{ \begin{array}{c}
k\in\mathscr{B}_{\rho}\ :\text{ the critical points }\xi\in I(\bar{\xi},R)\subset\partial M\\
\text{ of the mean curvature of \ensuremath{\partial M}are non degenerate}
\end{array}\right\} 
\end{eqnarray*}
is a residual subset of $\mathscr{B}_{\rho}$. Since $M$ is compact,
there exists a finite covering $\left\{ I(\xi_{i},R)\right\} _{i=1,\dots l}$
of $\partial M$, where $\xi_{1},\dots\xi_{l}\in\partial M$. For
any index $i$ there exists a residual set $A(\xi_{i},\rho)\subset\mathscr{B}_{\rho}$
such that the critical points of the curvature in $I(\xi_{i},R)$
are non degenerate for any $k\in A(\xi_{i},\rho)$. Let 
\[
A(\rho)=\cap_{i=1}^{l}A(\xi_{i},\rho).
\]
Then the set $A(\rho)$ is a residual set in $\mathscr{B}_{\rho}$.

We now may conclude that, given the metric $\bar{g}$, for any $\rho$
small enough there exists a $\bar{k}\in A(\rho)\subset\mathscr{B}_{\rho}$
such that the critical points of the mean curvature of $\partial M$
related to the metric $\bar{g}+\bar{k}$ are non degenerate. Thus
the set $\mathscr{A}$ defined in (\ref{eq:A}) is a dense set. Moreover,
by Lemma \ref{lem:aperto} we have that $\mathscr{A}$ is open and
the proof of Theorem \ref{thm:main} is complete.

\section{Technical lemmas}

We now prove two technical lemmas in order to obtain assumption (ii)
of the transversality theorem.
\begin{lem}
\label{lem:derivate}For any $x\in B(0,R)\subset\mathbb{R}^{n}$ and
for any $\tilde{k}\in\mathscr{B}_{\rho}\subset\mathscr{S}^{m}$ it
holds
\begin{eqnarray}
F'_{k}(\tilde{k},x)[k] & = & 
\left.(\tilde{g}^{-1}k\tilde{g}^{-1})_{im}\tilde{g}_{ml,s}\tilde{g}^{li}\tilde{g}_{ij,t}\right|_{x,0}
+\left.\tilde{g}^{im}k_{ml,s}\tilde{g}^{lj}\tilde{g}_{ij,t}\right|_{x,0}
+\left.\tilde{g}^{im}\tilde{g}_{ml,s}(\tilde{g}^{-1}k\tilde{g}^{-1})_{lj}\tilde{g}_{ij,t}\right|_{x,0}\nonumber \\
 &  & -\left.\tilde{g}^{im}\tilde{g}_{ml,s}\tilde{g}^{lj}k_{ij,t}\right|_{x,0}
 -\left.(\tilde{g}^{-1}k\tilde{g}^{-1})_{ij}\tilde{g}_{ij,ts}\right|_{x,0}+\left.\tilde{g}^{ij}k_{ij,ts}\right|_{x,0}\label{eq:*}
\end{eqnarray}
\end{lem}
\begin{proof}
We differentiate the identity $g^{im}g_{mj}=\delta_{ij}$, obtaining
\[
\partial_{x_{s}}g^{ij}:=g_{,s}^{ij}=-g^{im}g_{ml,s}g^{lj}
\]
Then we have 
\begin{multline}
\partial_{x_{s}}\left[\left(g+k\right)^{ij}(x,t)\left(g+k\right)_{ij,t}(x,t)\right]=\\
=-\left(g+k\right)^{im}\left(g+k\right)_{ml,s}\left(g+k\right)^{lj}\left(g+k\right)_{ij,t}
+\left(g+k\right)^{ij}\left(g+k\right)_{ij,ts}.\label{eq:gij}
\end{multline}
Here $\partial_{t}g_{ij}:=g_{ij,t}$. We recall that, if $\rho$ is
sufficently small, for any pair $k,\tilde{k}\in\mathscr{B}_{\rho}$,
we have 
\begin{equation}
\left(g+\tilde{k}+k\right)^{-1}=\left(\tilde{g}+k\right)
=\tilde{g}^{-1}-\tilde{g}^{-1}k\tilde{g}^{-1}
+\sum_{\lambda=2}^{\infty}(-1)^{\lambda}(\tilde{g}^{-1}k)^{\lambda}\tilde{g}^{-1}.\label{eq:inv}
\end{equation}
Here $\tilde{g}=g+\tilde{k}$. At this point, by (\ref{eq:gij}) and
by (\ref{eq:inv}), we have 
\begin{multline*}
D_{k}\partial_{x_{s}}\left.\left(\left(g+k\right)^{ij}(x,t)\left(g+k\right)_{ij,t}(x,t)\right)\right|_{\tilde{k},x,0}[k]=\\
=D_{\tilde{g}}\left.\left(\left(\tilde{g}\right)^{ij}(x,t)\left(\tilde{g}\right)_{ij,t}(x,t)\right)\right|_{\tilde{g},x,0}[k]=\\
=\left.(\tilde{g}^{-1}k\tilde{g}^{-1})_{im}\tilde{g}_{ml,s}\tilde{g}^{li}\tilde{g}_{ij,t}\right|_{x,0}
+\left.\tilde{g}^{im}k_{ml,s}\tilde{g}^{lj}\tilde{g}_{ij,t}\right|_{x,0}
+\left.\tilde{g}^{im}\tilde{g}_{ml,s}(\tilde{g}^{-1}k\tilde{g}^{-1})_{lj}\tilde{g}_{ij,t}\right|_{x,0}\\
-\left.\tilde{g}^{im}\tilde{g}_{ml,s}\tilde{g}^{lj}k_{ij,t}\right|_{x,0}
-\left.(\tilde{g}^{-1}k\tilde{g}^{-1})_{ij}\tilde{g}_{ij,ts}\right|_{x,0}+\left.\tilde{g}^{ij}k_{ij,ts}\right|_{x,0}.
\end{multline*}
This concludes the proof\end{proof}
\begin{lem}
\label{lem:ii}For any $(\tilde{x},\tilde{k})$ such that $F(\tilde{k},\tilde{x})=0$
we have that the map
\[
(x,k)\mapsto F'_{k}(\tilde{k},\tilde{x})[k]+F'_{x}(\tilde{k},\tilde{x})x
\]
 is onto on $\mathbb{R}^{n}$.\end{lem}
\begin{proof}
Let $(\tilde{x},\tilde{k})$ such that $F(\tilde{k},\tilde{x})=0$.
To obtain our claim it is sufficient to prove that the map 
$F'_{k}(\tilde{k},\tilde{x})[k]:\mathscr{S}^{m}\rightarrow\mathbb{R}^{n}$
is onto. More precisely, we are going to prove that, given $e_{1},\dots,e_{n}$
the canonical base in $\mathbb{R}^{n}$, for any $\nu=1,\dots,n$
there exists $k\in\mathscr{S}^{m}$ such that $F'_{k}(\tilde{k},\tilde{x})[k]=e_{\nu}.$
We remark that the ontoness is independent from the choice of the
coordinates, so, for any $\tilde{\xi}=(\tilde{x},0)\in\partial M$
we choose the exponential coordinates in $\partial M$ with metric
$g+\tilde{k}$ centered in $\tilde{\xi}$, so we have to prove simply
that, given $\nu$, there exists $k\in\mathscr{S}^{m}$ such that
\begin{eqnarray*}
D_{k}\partial_{x_{\nu}}\left.\left(\left(g+k\right)^{ij}(x,t)\left(g+k\right)_{ij,t}(x,t)\right)\right|_{\tilde{k},x,0}[k] & = & 1\\
D_{k}\partial_{x_{s}}\left.\left(\left(g+k\right)^{ij}(x,t)\left(g+k\right)_{ij,t}(x,t)\right)\right|_{\tilde{k},x,0}[k] & = & 0
\end{eqnarray*}
for all $s\ne\nu$. 

We use (\ref{eq:*}) of Lemma \ref{lem:derivate}. Using Fermi coordinates
we have that the metric $\tilde{g}(\cdot,0)$ on the submanifold $\partial M$
has the form $\tilde{g}_{ij}(0,0)=\delta_{ij},$ thus 
\begin{multline}
D_{k}\partial_{x_{s}}\left.\left(\left(g+k\right)^{ij}(x,t)\left(g+k\right)_{ij,t}(x,t)\right)\right|_{\tilde{k},0,0}[k]=\\
\left.k_{im}\tilde{g}_{mj,s}\tilde{g}_{ij,t}\right|_{0,0}
-\left.k_{ij,s}\tilde{g}_{ij,t}\right|_{0,0}+\left.\tilde{g}_{il,s}k_{lj}\tilde{g}_{ij,t}\right|_{0,0}\\
-\left.\tilde{g}_{ij,s}k_{ij,t}\right|_{0,0}-\left.k_{ij}\tilde{g}_{ij,ts}\right|_{0,0}+\left.k_{ii,ts}\right|_{0,0}\label{eq:**}
\end{multline}

Now, we choose $k\in\mathscr{S}^{m}$ such that the map $x\mapsto k_{ij}(x,0)$
vanishes at $x=0$ for all $i,j$. Then by (\ref{eq:**}) we have
\begin{multline*}
D_{k}\partial_{x_{s}}\left.\left(\left(g+k\right)^{ij}(x,t)\left(g+k\right)_{ij,t}(x,t)\right)\right|_{\tilde{k},0,0}[k]=\\
-\left.k_{ij,s}\tilde{g}_{ij,t}\right|_{0,0}-\left.\tilde{g}_{ij,s}k_{ij,t}\right|_{0,0}+\left.k_{ii,ts}\right|_{0,0}
\end{multline*}
Moreover we assume the the first derivatives $k_{ij,s}(0,0)=k_{ij,t}(0,0)=0$
for all $i,j$. Then
\begin{equation}
D_{k}\partial_{x_{s}}\left.\left(\left(g+k\right)^{ij}(x,t)\left(g+k\right)_{ij,t}(x,t)\right)\right|_{\tilde{k},0,0}[k]
=\left.k_{ii,ts}\right|_{0,0}.\label{eq:****}
\end{equation}
We now prove the claim for $\nu=1$. Let us choose $k\in\mathscr{S}^{m}$
such that 
\begin{eqnarray*}
k_{11}(x,t)=x_{1}t; &  & k_{ij}(x,t)=0
\end{eqnarray*}
 for $(i,j)\ne(1,1)$. Thus (\ref{eq:****}) rewrites as 
\begin{eqnarray*}
D_{k}\partial_{x_{1}}\left.\left(\left(g+k\right)^{ij}(x,t)\left(g+k\right)_{ij,t}(x,t)\right)\right|_{\tilde{k},0,0}[k] 
& = & \left.k_{11,ts}\right|_{0,0}=1\\
D_{k}\partial_{x_{s}}\left.\left(\left(g+k\right)^{ij}(x,t)\left(g+k\right)_{ij,t}(x,t)\right)\right|_{\tilde{k},0,0}[k] 
& = & \left.k_{ss,ts}\right|_{0,0}=0
\end{eqnarray*}
for $s=2,\dots,n$. Analogously we proceed for any $\nu=2,\dots,n$
and we conclude the proof.\end{proof}

\end{document}